\documentclass{amsart}

\usepackage{amssymb,amsmath,amsthm}
\usepackage{graphicx}
\usepackage{enumitem}
\usepackage{booktabs}
\usepackage{array}
\usepackage{caption}
\usepackage{float}
\usepackage{xcolor}
\usepackage{tcolorbox}
\usepackage{hyperref}

\definecolor{mainblue}{RGB}{15,45,90}
\definecolor{accent}{RGB}{170,30,45}

\hypersetup{
    colorlinks=true,
    linkcolor=mainblue,
    urlcolor=accent,
    citecolor=mainblue
}

\newtheorem{theorem}{Theorem}[section]
\newtheorem{lemma}[theorem]{Lemma}
\newtheorem{proposition}[theorem]{Proposition}
\newtheorem{corollary}[theorem]{Corollary}

\theoremstyle{definition}
\newtheorem{definition}[theorem]{Definition}
\newtheorem{example}[theorem]{Example}

\theoremstyle{remark}
\newtheorem{remark}[theorem]{Remark}

\numberwithin{equation}{section}

\begin{document}

\title[Finitely Related Clones]{Finitely Related Clones: Action Algebras and Applications to Free Algebras}

\author{Vishwesh Tiwari}
\address{Department of Mathematics, Indian Institute of Science Education and Research (IISER) Pune, Maharashtra, India}
\email{vishwesh.tiwari@students.iiserpune.ac.in}



\date{}

\begin{abstract}
For finite-domain constraint satisfaction problems (CSPs), the
polymorphism clone provides an algebraic framework for studying
the associated constraint language. A central question in clone
theory is whether a clone is finitely related, that is, whether
it is determined by the operations preserving a finite set of
relations.

Here, we study this question through the action algebra
of a clone on its own $n$-ary part. We prove that a clone
$\mathcal{C}$ on a finite domain $C$ is finitely related if and
only if its action algebra
$\widetilde{\mathcal{C}} = \mathcal{C} \curvearrowright
\mathcal{C}^{(n)}$ is finitely related, provided that
$n \geq |C|$. We also apply natural clone isomorphisms to
obtain a finite-relatedness criterion for free algebras:
a finite algebra is finitely related if and only if its free
algebra of sufficiently large finite rank is finitely related.
\end{abstract}

\maketitle

\section{Introduction}

The algebraic approach to constraint satisfaction problems
(CSPs) studies the operations that preserve the relations of a
constraint language. These operations, called
\emph{polymorphisms}, form a clone and provide an algebraic
framework for understanding the associated CSP.

The connection between relational structures and polymorphisms
was established by Geiger \cite{geiger1968} and Bodnarchuk
et al. \cite{bodnarchuk1969}. This framework has played an
important role in the study of CSP complexity, from Schaefer's
Dichotomy Theorem for Boolean domains \cite{schaefer1978}
to the resolution of the general CSP Dichotomy Conjecture
\cite{federvardi, bulatov2017, zhuk2017}.

A central question in clone theory is whether a clone is
\emph{finitely related}, meaning that it is precisely the
set of operations preserving some finite collection of
relations. Since a clone contains operations of arbitrarily
large arities, finite relatedness connects an infinite
collection of operations to a finite relational description.

\textbf{Contributions.} The primary contribution of this paper is the reduction of this infinite structural problem into a finite operational setting using the machinery of clone actions and induced algebras. Specifically, we define the \emph{action algebra} $\widetilde{\mathcal{C}} = \mathcal{C} \curvearrowright \mathcal{C}^{(n)}$, where the clone acts upon its own $n$-ary part. Our main theorem establishes that a clone $\mathcal{C}$ on a finite domain $C$ is finitely related if and only if its action algebra $\widetilde{\mathcal{C}}$ is finitely related for $n \ge |C|$. This reduces the question of finite relatedness of the clone to the finite action algebra $\mathcal{C} \curvearrowright \mathcal{C}^{(n)}$. The foundational idea behind this reduction was proposed by Michael Pinsker; building upon his suggestion, we provide the detailed algebraic formulation and proofs.

As an application, we use natural clone isomorphisms to
derive a finite- \linebreak relatedness criterion for free algebras.
In particular, for a finite algebra $\mathbf{A}$ and
sufficiently large finite rank $n$, we establish an
equivalence between the finite relatedness of $\mathbf{A}$
and that of the free algebra
$\mathbf{F}_n(\mathcal{V}(\mathbf{A}))$.
This criterion is related to the result of Davey, Jackson,
Pitkethly, and Szabó \cite{davey2011}.

\textbf{Organization.} The paper is organized as follows. Section 2 introduces the necessary preliminaries on relational structures, polymorphisms, and primitive positive definability. Section 3 develops the theory of finitely related clones, proves the main characterization via action algebras, and derives the application to free algebras.

\section{Preliminaries: Relational Structures and Polymorphisms}

We begin with the basic definitions of relational structures, constraint satisfaction problems, and polymorphisms.

\begin{definition}[]
A \emph{relational structure} is a tuple $\mathcal{A} = (A; R_1, \dots, R_k)$, where $A$ is a non-empty set called the \emph{domain}, and each $R_i \subseteq A^{n_i}$ is an $n_i$-ary relation on $A$.
\end{definition}

\begin{definition}[]
Let $\mathcal{A} = (A; R_1, \dots, R_k)$ be a fixed finite relational structure. The constraint satisfaction problem parameterized by $\mathcal{A}$, denoted $\text{CSP}(\mathcal{A})$, takes as input a set of variables $V$ and a set of constraints $(s_j, R_{i_j})$ where $s_j$ is a tuple of variables. The problem asks whether there exists an assignment $h: V \to A$ such that every constraint is satisfied (i.e., $h(s_j) \in R_{i_j}$).
\end{definition}

We now introduce polymorphisms, which connect CSPs with universal algebra.

\begin{definition}[]
Let $R \subseteq A^n$ be an $n$-ary relation. A $k$-ary operation $f: A^k \rightarrow A$ is said to \emph{preserve} $R$ (or is a \emph{polymorphism} of $R$) if, whenever we take $k$ tuples $r^1, \dots, r^k \in R$, applying $f$ coordinatewise yields a new tuple that also belongs to $R$. That is:
\[
f(r^1, \dots, r^k) \in R.
\]
For a relational structure $\mathcal{A}$, the set of all operations that preserve every relation in $\mathcal{A}$ is denoted by $\text{Pol}(\mathcal{A})$. Conversely, for a set of operations $F$ on $A$, the set of all relations preserved by every operation in $F$ is denoted by $\text{Inv}(F)$.
\end{definition}

We can define new relations from existing ones using a restricted form of first-order logic.

\begin{definition}[]
A relation $R \subseteq A^n$ is \emph{primitively positively definable} (or \emph{pp-definable}) from a relational structure $\mathcal{A}$ if $R$ can be defined using a first-order formula over $\mathcal{A}$ restricted to existential quantifiers ($\exists$), conjunctions ($\land$), equality ($=$), and the always false formula ($\bot$). The set of all such relations is denoted $\text{PP}(\mathcal{A})$.
\end{definition}

If a constraint language is pp-definable from an existing one, any problem in the new language can be translated into a problem in the base language \cite{jeavons1998}. The following theorem describes the connection between pp-definability and polymorphisms: the relations pp-definable from $\mathcal{A}$ are precisely those preserved by all polymorphisms of $\mathcal{A}$.

\begin{theorem}[Geiger \cite{geiger1968}, Bodnarchuk et al. \cite{bodnarchuk1969}]
\label{thm:geiger}
Let $\mathcal{A}$ be a finite relational structure. Then:
\[
\text{PP}(\mathcal{A}) = \text{Inv}(\text{Pol}(\mathcal{A})).
\]
\end{theorem}

Thus, the polymorphisms of $\mathcal{A}$ play a central role in studying the complexity of $\text{CSP}(\mathcal{A})$. The set of polymorphisms forms an algebraic structure known as a clone.

\begin{definition}[]
\label{def:clone}
Let $A$ be a set. A \emph{clone} on $A$ is a set $\mathcal{C}$ of finitary operations on $A$ that satisfies two fundamental properties:
\begin{enumerate}
    \item \textit{Contains Projections:} For every arity $n \ge 1$ and index $1 \le i \le n$, the $n$-ary projection operation $\pi_i^n(x_1, \dots, x_n) = x_i$ belongs to $\mathcal{C}$.
    \item \textit{Closed under Composition:} If $f \in \mathcal{C}$ is an $m$-ary operation, and $g_1, \dots, g_m \in \mathcal{C}$ are $n$-ary operations, their generalized composition
    \[
    h(x_1, \dots, x_n) = f(g_1(x_1, \dots, x_n), \dots, g_m(x_1, \dots, x_n))
    \]
    also belongs to $\mathcal{C}$.
\end{enumerate}
\end{definition}

\section{Finitely Related Clones}

In this section, we study finitely related clones and prove a central theorem: a clone is finitely related if and only if it is determined by its $n$-ary part for sufficiently large $n$. The main idea behind the proof was suggested by Michael Pinsker, and we develop the algebraic details here.

\begin{definition}[]
The \emph{$n$-ary part} of a clone $\mathcal{C}$ is defined as the set of all $n$-ary operations belonging to $\mathcal{C}$:
\[
\mathcal{C}^{(n)} = \{ f \in \mathcal{C} \mid \text{arity}(f) = n \}.
\]
\end{definition}

In particular, $\mathcal{C}^{(1)}$ represents the set of all unary operations in the clone, $\mathcal{C}^{(2)}$ represents the binary operations, and so forth.

\begin{example}
\label{ex:nary_part}
Suppose $C = \{0,1\}$ and $\mathcal{C}$ is the Boolean clone generated by the logical AND operation $\wedge$.
\begin{enumerate}
    \item $\mathcal{C}^{(1)} = \{\pi_1^1\}$ contains only the unary projection.
    \item $\mathcal{C}^{(2)} = \{\pi_1^2, \pi_2^2, \wedge\}$ contains the two binary projections and the meet operation itself.
\end{enumerate}
\end{example}

\subsection{Characterizing Finitely Related Clones}

We now introduce the concept of a finitely related clone.

\begin{definition}[]
\label{def:finitely_related_clone}
A clone $\mathcal{C}$ is \emph{finitely related} if $\mathcal{C} = \text{Pol}(R_1, \dots, R_k)$ for finitely many relations $R_1, \dots, R_k$.
\end{definition}

\begin{proposition}
\label{prop:single_relation}
On a finite domain, a clone $\mathcal{C}$ is finitely related if and only if there exists a single relation $R$ such that $\mathcal{C} = \text{Pol}(R)$.
\end{proposition}

\begin{proof}
Let $\mathcal{C} = \text{Pol}(R_1, \dots, R_k)$ be finitely related. We may assume without loss of generality that each $R_i$ is non-empty, since an empty relation imposes no restriction on polymorphisms and may therefore be omitted from the generating set. Let the arities of $R_1, \dots, R_k$ be $m_1, \dots, m_k$ respectively. We construct a single, combined relation $R$ by taking their Cartesian product:
\[
R = R_1 \times R_2 \times \dots \times R_k.
\]
The new relation $R$ has arity $m = m_1 + \dots + m_k$. By definition, a tuple belongs to $R$ if and only if its first $m_1$ coordinates form a tuple in $R_1$, the next $m_2$ coordinates form a tuple in $R_2$, and so on.

When we apply any operation $f$ coordinatewise to tuples in $R$, the operation acts completely independently on each of these coordinate blocks. Therefore, $f$ will output a valid tuple in $R$ if and only if it simultaneously outputs a valid tuple in every individual block $R_i$. 

This means $f$ preserves the combined relation $R$ if and only if it preserves every individual relation $R_i$. Thus, $\text{Pol}(R_1, \dots, R_k) = \text{Pol}(R)$.
\end{proof}

We now connect the property of being finitely related to relational constructions. To do this, we must first explain how a set of operations can be formally treated as a relation.

Originally, we defined the $m$-ary part of our clone, $\mathcal{C}^{(m)}$, as a set of $m$-ary operations. However, because the base set $C$ is finite, we can convert these operations into tuples. Let $N = |C|^m$. We can systematically enumerate all possible $m$-tuples in the domain as $t_1, \dots, t_N \in C^m$. Every $m$-ary operation $f: C^m \to C$ can then be uniquely identified with its value table:
\[
\big(f(t_1), \dots, f(t_N)\big) \in C^N.
\]
Consequently, every $m$-ary operation becomes a single $N$-tuple. This allows us to view the entire set of operations $\mathcal{C}^{(m)}$ as a subset of $C^N$:
\[
\mathcal{C}^{(m)} \subseteq C^N.
\]
That is, $\mathcal{C}^{(m)}$ is now viewed as a finite relation of arity $N$. Therefore, the notation $\text{Pol}(\mathcal{C}^{(m)})$ makes rigorous sense: it denotes the clone of all operations preserving this newly constructed relation.

\begin{example}
\label{ex:c_m_as_relation}
Consider the case where $m = 2$ and $C = \{0, 1\}$. The set of binary inputs $C^2$ consists of the following 4 tuples:
\[
C^2 = \{(0,0), (0,1), (1,0), (1,1)\}.
\]
Let $f: C^2 \rightarrow C$ be defined as the meet operation $f(x,y) = x \wedge y$. It takes the values $0, 0, 0, 1$ respectively on these four enumerated inputs. Thus, $f$ is completely described by the tuple:
\[
(0, 0, 0, 1) \in C^4.
\]
We can therefore view this operation as a single point in $C^4$. 

Now, suppose the binary part of our clone is $\mathcal{C}^{(2)} = \{\pi_1, \pi_2, \wedge\}$. After replacing each operation with its corresponding value table tuple, we obtain:
\[
\mathcal{C}^{(2)} = \{(0,0,1,1), (0,1,0,1), (0,0,0,1)\}.
\]
This set is simply a subset of $C^4$, which is a $4$-ary relation on $C$. In this manner, $\mathcal{C}^{(m)}$ becomes a relation, allowing us to consider its polymorphism clone $\text{Pol}\left(\mathcal{C}^{(m)}\right)$.
\end{example}

With this understanding, we can apply the foundational results of universal algebra. We restate a specialized, quantitative version of Theorem~\ref{thm:geiger}. We restate it here because this specific formulation bounds the arity; it tells us exactly \emph{which} finite relation ($\mathcal{C}^{(m)}$) we need to use as our base to pp-define any other relation $R$, provided $m$ is large enough.

\begin{theorem}[Bodnarchuk et al. \cite{bodnarchuk1969}, Geiger \cite{geiger1968}]
\label{thm:geiger_quantitative}
Let $C$ be a finite set, $R \subseteq C^r$ be a relation, and $\mathcal{C} = \text{Pol}(R)$. If we choose an arity $m \ge |R|$, then $R$ is primitive positively definable from the relation $\mathcal{C}^{(m)}$, denoted by:
\[
R \in \text{PP}\left(\mathcal{C}^{(m)}\right).
\]
\end{theorem}

\begin{theorem}
\label{thm:finitely_related_nary}
A clone $\mathcal{C}$ on a finite set $C$ is finitely related if and only if there exists an arity $m_0 \ge 1$ such that for all $m \ge m_0$:
\[
\mathcal{C} = \text{Pol}\left(\mathcal{C}^{(m)}\right).
\]
\end{theorem}

\begin{proof}
$(\implies)$ Assume that the clone $\mathcal{C}$ is finitely related. By Proposition~\ref{prop:single_relation}, there exists a single finite relation $R \subseteq C^r$ such that $\mathcal{C} = \text{Pol}(R)$. Let us set $m_0 := |R|$. We claim that for every arity $m \ge m_0$, we have $\mathcal{C} = \text{Pol}\left(\mathcal{C}^{(m)}\right)$. We prove this equality by demonstrating inclusion in both directions.

To show that $\mathcal{C} \subseteq \text{Pol}\left(\mathcal{C}^{(m)}\right)$, take an arbitrary $k$-ary operation $f \in \mathcal{C}$. We must verify that $f$ preserves the relation $\mathcal{C}^{(m)}$. Recall that elements of this relation are exactly the value tables of $m$-ary operations in $\mathcal{C}$. Take $k$ arbitrary tuples from the relation $\mathcal{C}^{(m)}$. These tuples correspond to the value tables of some operations $g_1, \dots, g_k \in \mathcal{C}^{(m)}$ evaluated over the enumeration of the domain $t_1, \dots, t_N$:
\[
\big(g_i(t_1), \dots, g_i(t_N)\big) \quad \text{for } i \in \{1, \dots, k\}.
\]
When we apply $f$ coordinatewise to these $k$ tuples, the $j$-th coordinate of the resulting tuple is:
\[
f\big(g_1(t_j), \dots, g_k(t_j)\big).
\]
This resulting tuple is precisely the value table of the composite operation 
\[
h(x_1, \dots, x_m) = f(g_1(x_1, \dots, x_m), \dots, g_k(x_1, \dots, x_m)).
\]
Because $\mathcal{C}$ is a clone, it is closed under composition, meaning $h \in \mathcal{C}$. Furthermore, since each $g_i$ is $m$-ary, the composition $h$ is also $m$-ary, so $h \in \mathcal{C}^{(m)}$. Thus, the resulting tuple belongs to the relation $\mathcal{C}^{(m)}$, proving that $f$ preserves the relation. Hence, $\mathcal{C} \subseteq \text{Pol}\left(\mathcal{C}^{(m)}\right)$.

For the reverse inclusion, $\text{Pol}\left(\mathcal{C}^{(m)}\right) \subseteq \mathcal{C}$, recall that $m \ge m_0 = |R|$. By Theorem~\ref{thm:geiger_quantitative}, the relation $R$ is primitive positively definable from $\mathcal{C}^{(m)}$, meaning $R \in \text{PP}\left(\mathcal{C}^{(m)}\right)$. Now, take an arbitrary operation $f \in \text{Pol}\left(\mathcal{C}^{(m)}\right)$. Because primitive positively definable relations are inherently preserved by every polymorphism of their defining relations, $f$ must also preserve $R$. This implies $f \in \text{Pol}(R)$. Since our initial assumption states that $\text{Pol}(R) = \mathcal{C}$, we directly obtain $f \in \mathcal{C}$. This establishes $\text{Pol}\left(\mathcal{C}^{(m)}\right) \subseteq \mathcal{C}$.

Combining both directions, we conclude that $\mathcal{C} = \text{Pol}\left(\mathcal{C}^{(m)}\right)$ for all $m \ge m_0$.

$(\impliedby)$ Suppose there exists an arity $m_0 \ge 1$ such that $\mathcal{C} = \text{Pol}\left(\mathcal{C}^{(m)}\right)$ for all $m \ge m_0$. By fixing any such $m$, we have expressed $\mathcal{C}$ as the clone of polymorphisms of a specific relation, namely $\mathcal{C}^{(m)}$. Because the base set $C$ is finite, the relation $\mathcal{C}^{(m)} \subseteq C^{|C|^m}$ is a finite relation. By definition, a clone determined by a finite relation is finitely related, completing the proof.
\end{proof}

\subsection{Clone Actions and Induced Algebras}

Let $\mathcal{C}$ be a clone on a finite set $C$. Consider an operation $f \in \mathcal{C}$ of arity $k$. Recall that $\mathcal{C}^{(n)}$ denotes the $n$-ary part of the clone, defined by $\mathcal{C}^{(n)} = \{g : C^n \rightarrow C \mid g \in \mathcal{C}\}$. 

Since $\mathcal{C}$ is a clone, it is closed under composition. Therefore, for any $n$-ary operations $f_1, \dots, f_k \in \mathcal{C}^{(n)}$, the composition $f(f_1, \dots, f_k)$ is again an $n$-ary operation in $\mathcal{C}^{(n)}$. Consequently, $f$ acts directly on the set $\mathcal{C}^{(n)}$ by mapping these functions to a new function. From this construction, one can observe that every operation $f \in \mathcal{C}$ induces an operation on the set $\mathcal{C}^{(n)}$. 

\begin{definition}[]
\label{def:action_algebra}
In universal algebra, an \emph{algebra} is a pair consisting of a non-empty universe and a set of basic operations. Where $\mathcal{C}$ is a clone on $C$, we define the \emph{action algebra} $\widetilde{\mathcal{C}} := \mathcal{C} \curvearrowright \mathcal{C}^{(n)}$, whose universe is the set of $n$-ary operations $\mathcal{C}^{(n)}$, and whose operations are the actions induced by the elements of $\mathcal{C}$. 
\end{definition}

Before proving our main theorem regarding the finite relatedness of this action algebra, we must establish two structural properties. First, we examine the full clone of all operations on $C$, denoted $\mathcal{O}$, and its corresponding full action algebra $\widetilde{\mathcal{O}} := \mathcal{O} \curvearrowright \mathcal{O}^{(n)}$. The universe of this algebra is $A = \mathcal{O}^{(n)} = C^{C^n}$. Every $l$-ary operation $f \in \mathcal{O}$ induces an operation $\tilde{f} : A^l \rightarrow A$, where the composition is defined pointwise:
\[
\tilde{f}(g_1, \dots, g_l)(a_1, \dots, a_n) = f\big(g_1(a_1, \dots, a_n), \dots, g_l(a_1, \dots, a_n)\big).
\]

\begin{lemma}
\label{lem:equality_relation}
For any $n \ge 1$, the full action algebra $\widetilde{\mathcal{O}} = \mathcal{O} \curvearrowright \mathcal{O}^{(n)}$ is finitely related. In particular, $\widetilde{\mathcal{O}} = \text{Pol}(S)$ for a single finite relation $S$.
\end{lemma}

\begin{proof}
Let $A = \mathcal{O}^{(n)} = C^{C^n}$ be the universe of the full action algebra $\widetilde{\mathcal{O}}$. For any pair of tuples $t, u \in C^n$, we define the binary relation $S_{t,u} \subseteq A^2$ by:
\[
S_{t,u} = \{ (g, h) \in A^2 \mid g(t) = h(u) \}.
\]
We claim that $\widetilde{\mathcal{O}} = \text{Pol}(\{S_{t,u} \mid t,u \in C^n\})$. Since $C^n$ is finite, there are finitely many such relations $S_{t,u}$. By Proposition~\ref{prop:single_relation}, these finitely many relations can be combined into a single finite relation $S$ such that $\widetilde{\mathcal{O}} = \text{Pol}(S)$.

We prove the claim by establishing inclusion in both directions.

\medskip
\noindent \textit{Direction 1: $\widetilde{\mathcal{O}} \subseteq \text{Pol}(\{S_{t,u} \mid t,u \in C^n\})$}

Take arbitrary tuples $t, u \in C^n$. Let $f \in \mathcal{O}$ be an $l$-ary operation on $C$, and let $\tilde{f} \in \widetilde{\mathcal{O}}$ be its induced $l$-ary action on $A$.

Suppose we take $l$ pairs from $S_{t,u}$, say $(g_1, g'_1), \dots, (g_l, g'_l) \in S_{t,u}$. By definition of $S_{t,u}$, this means that $g_i(t) = g'_i(u)$ for all $i \in \{1, \dots, l\}$.

Applying $\tilde{f}$ coordinatewise to these $l$ pairs yields the pair $(\tilde{f}(g_1, \dots, g_l), \linebreak \tilde{f}(g'_1, \dots, g'_l))$. Evaluating both components:
\begin{align*}
\tilde{f}(g_1, \dots, g_l)(t) &= f\big(g_1(t), \dots, g_l(t)\big) \\
&= f\big(g'_1(u), \dots, g'_l(u)\big) \\
&= \tilde{f}(g'_1, \dots, g'_l)(u).
\end{align*}
Thus, $(\tilde{f}(g_1, \dots, g_l), \tilde{f}(g'_1, \dots, g'_l)) \in S_{t,u}$, proving that $\tilde{f}$ preserves $S_{t,u}$ for all $t,u \in C^n$. Hence, $\widetilde{\mathcal{O}} \subseteq \text{Pol}(\{S_{t,u} \mid t,u \in C^n\})$.

\medskip
\noindent \textit{Direction 2: $\text{Pol}(\{S_{t,u} \mid t,u \in C^n\}) \subseteq \widetilde{\mathcal{O}}$}

Let $F : A^l \to A$ be an $l$-ary operation on $A$ that preserves $S_{t,u}$ for all $t, u \in C^n$. We must show that $F = \tilde{f}$ for some $l$-ary operation $f \in \mathcal{O}$ on $C$.

For any vector $s = (s_1, \dots, s_l) \in C^l$, we pick a tuple $t^s \in C^n$ and $n$-ary operations $h_1^s, \dots, h_l^s \in \mathcal{O}^{(n)}$ such that:
\[
h_j^s(t^s) = s_j \quad \text{for all } j \in \{1, \dots, l\}.
\]
(Such operations $h_j^s$ always exist; for instance, one can simply choose $h_j^s$ to be the constant function returning $s_j$).

We then define the function $f : C^l \to C$ by setting:
\[
f(s) := F\big(h_1^s, \dots, h_l^s\big)(t^s).
\]

\textit{Well-definedness of $f$:} Suppose we make a different choice of evaluation tuple $u^s \in C^n$ and operations $k_1^s, \dots, k_l^s \in \mathcal{O}^{(n)}$ satisfying $k_j^s(u^s) = s_j$ for all $j$. Then for each $j \in \{1, \dots, l\}$, we have $h_j^s(t^s) = s_j = k_j^s(u^s)$, which implies $(h_j^s, k_j^s) \in S_{t^s, u^s}$.

Since $F$ preserves $S_{t^s, u^s}$, applying $F$ to these pairs yields:
\[
\Big(F(h_1^s, \dots, h_l^s), F(k_1^s, \dots, k_l^s)\Big) \in S_{t^s, u^s},
\]
which by definition means:
\[
F(h_1^s, \dots, h_l^s)(t^s) = F(k_1^s, \dots, k_l^s)(u^s).
\]
Therefore, the value of $f(s)$ is independent of the choice of $t^s$ and $h_j^s$, so $f : C^l \to C$ is a well-defined operation in $\mathcal{O}$.

\textit{Showing $F = \tilde{f}$:} Take arbitrary operations $g_1, \dots, g_l \in \mathcal{O}^{(n)}$ and an arbitrary tuple $t \in C^n$. Set $s = (g_1(t), \dots, g_l(t)) \in C^l$.

Since $g_j(t) = s_j$ for all $j$, the tuple of pairs $(g_j, h_j^s)$ belongs to $S_{t, t^s}$ for every $j \in \{1, \dots, l\}$. Because $F$ preserves $S_{t, t^s}$, we have:
\[
F(g_1, \dots, g_l)(t) = F(h_1^s, \dots, h_l^s)(t^s).
\]
By our definition of $f(s)$, the right-hand side equals $f(s) = f(g_1(t), \dots, g_l(t))$. Thus:
\[
F(g_1, \dots, g_l)(t) = f(g_1(t), \dots, g_l(t)) = \tilde{f}(g_1, \dots, g_l)(t).
\]
Since this holds for all $g_1, \dots, g_l \in \mathcal{O}^{(n)}$ and all $t \in C^n$, we conclude that $F = \tilde{f} \in \widetilde{\mathcal{O}}$.

\medskip
This proves $\text{Pol}(\{S_{t,u} \mid t,u \in C^n\}) = \widetilde{\mathcal{O}}$, completing the proof.
\end{proof}

\begin{lemma}
\label{lem:action_homomorphism}
The mapping $f \mapsto \tilde{f}$ from a clone to its action algebra acts as a homomorphism with respect to composition. Specifically, for any $k$-ary operation $f \in \mathcal{C}$ and any $m$-ary operations $g_1, \dots, g_k \in \mathcal{C}$, we have:
\[
\widetilde{f(g_1, \dots, g_k)} = \tilde{f}(\tilde{g}_1, \dots, \tilde{g}_k).
\]
\end{lemma}

\begin{proof}
We evaluate both sides on an arbitrary $n$-ary operation $h \in \mathcal{C}^{(n)}$. By definition of clone action, the left side gives:
\[
\big(\widetilde{f(g_1, \dots, g_k)}\big)(h) = f(g_1, \dots, g_k)(h) = f\big(g_1(h), \dots, g_k(h)\big).
\]
For the right side, applying $\tilde{f}$ to $\tilde{g}_1(h), \dots, \tilde{g}_k(h)$ yields:
\[
\tilde{f}(\tilde{g}_1, \dots, \tilde{g}_k)(h) = f\big(\tilde{g}_1(h), \dots, \tilde{g}_k(h)\big) = f\big(g_1(h), \dots, g_k(h)\big).
\]
Thus, the two operations are identical.
\end{proof}

Theorem~\ref{thm:action_algebra_finitely_related} gives the following characterization.

\begin{theorem}
\label{thm:action_algebra_finitely_related}
For $n \ge |C|$, the clone $\mathcal{C}$ is finitely related if and only if its action algebra $\widetilde{\mathcal{C}} = \mathcal{C} \curvearrowright \mathcal{C}^{(n)}$ is finitely related.
\end{theorem}

\begin{proof}
$(\implies)$ Assume that the clone $\mathcal{C}$ is finitely related. By Theorem~\ref{thm:finitely_related_nary}, there exists an arity $m \ge 1$ such that $\mathcal{C} = \text{Pol}\left(\mathcal{C}^{(m)}\right)$.

Fix $n \ge |C|$, and let $A = \mathcal{C}^{(n)}$ be the universe of the action algebra $\widetilde{\mathcal{C}}$. We define $\widetilde{\mathcal{C}}^{(m)} \subseteq A^{A^m}$ as the relation formed by the induced actions of all $m$-ary operations in $\mathcal{C}$:
\[
\widetilde{\mathcal{C}}^{(m)} = \{ \tilde{g} : A^m \to A \mid g \in \mathcal{C}^{(m)} \}.
\]
We will show that the action algebra $\widetilde{\mathcal{C}}$ is finitely related by proving that $\widetilde{\mathcal{C}} = \text{Pol}(S, \widetilde{\mathcal{C}}^{(m)})$, where $S$ is the finite relation from Lemma~\ref{lem:equality_relation}. We prove this equality by showing inclusion in both directions.

First, we show $\widetilde{\mathcal{C}} \subseteq \text{Pol}(S, \widetilde{\mathcal{C}}^{(m)})$. Take an arbitrary operation $\tilde{f} \in \widetilde{\mathcal{C}}$. By definition, $\tilde{f}$ is the action induced by some $k$-ary operation $f \in \mathcal{C}$. By Lemma~\ref{lem:equality_relation}, $\tilde{f}$ preserves $S$. To see that $\tilde{f}$ preserves $\widetilde{\mathcal{C}}^{(m)}$, take $k$ arbitrary actions $\tilde{g}_1, \dots, \tilde{g}_k \in \widetilde{\mathcal{C}}^{(m)}$. By Lemma~\ref{lem:action_homomorphism}, $\tilde{f}(\tilde{g}_1, \dots, \tilde{g}_k) = \widetilde{f(g_1, \dots, g_k)}$. Since $f \in \mathcal{C}$ and $g_1, \dots, g_k \in \mathcal{C}^{(m)}$, their composition $f(g_1, \dots, g_k) \in \mathcal{C}^{(m)}$, so its induced action belongs to $\widetilde{\mathcal{C}}^{(m)}$. Thus, $\tilde{f} \in \text{Pol}(S, \widetilde{\mathcal{C}}^{(m)})$.

Conversely, we show $\text{Pol}(S, \widetilde{\mathcal{C}}^{(m)}) \subseteq \widetilde{\mathcal{C}}$. Let $F \in \text{Pol}(S, \widetilde{\mathcal{C}}^{(m)})$. Since $F$ preserves $S$, Lemma~\ref{lem:equality_relation} guarantees that $F = \tilde{f}$ for some $k$-ary operation $f \in \mathcal{O}$. To show $f \in \mathcal{C} = \text{Pol}\left(\mathcal{C}^{(m)}\right)$, take arbitrary operations $g_1, \dots, g_k \in \mathcal{C}^{(m)}$. Since $F = \tilde{f}$ preserves $\widetilde{\mathcal{C}}^{(m)}$, we have $\tilde{f}(\tilde{g}_1, \dots, \tilde{g}_k) \in \widetilde{\mathcal{C}}^{(m)}$. By Lemma~\ref{lem:action_homomorphism}, this means $\widetilde{f(g_1, \dots, g_k)} \in \widetilde{\mathcal{C}}^{(m)}$, which implies $f(g_1, \dots, g_k) \in \mathcal{C}^{(m)}$. Hence, $f$ preserves $\mathcal{C}^{(m)}$, so $f \in \mathcal{C}$ and $F = \tilde{f} \in \widetilde{\mathcal{C}}$.

This establishes $\widetilde{\mathcal{C}} = \text{Pol}(S, \widetilde{\mathcal{C}}^{(m)})$, so the action algebra $\widetilde{\mathcal{C}}$ is finitely related.

\medskip
$(\impliedby)$ Assume that the action algebra $\widetilde{\mathcal{C}} = \mathcal{C} \curvearrowright \mathcal{C}^{(n)}$ is finitely related. By Proposition~\ref{prop:single_relation}, there exists a single finite relation $R$ on the universe $A = \mathcal{C}^{(n)}$ such that $\widetilde{\mathcal{C}} = \text{Pol}(R)$.

Choose an arity $m \ge |R|$. (Note that this arity $m$ depends on $|R|$ and may be larger than the arity $m$ used in the forward direction). By Theorem~\ref{thm:geiger_quantitative} applied to the action algebra $\widetilde{\mathcal{C}}$, the relation $R$ is primitive positively definable from the relation $\widetilde{\mathcal{C}}^{(m)}$, that is, $R \in \text{PP}(\widetilde{\mathcal{C}}^{(m)})$. Consequently, any operation that preserves $\widetilde{\mathcal{C}}^{(m)}$ must also preserve $R$:
\[
\text{Pol}(\widetilde{\mathcal{C}}^{(m)}) \subseteq \text{Pol}(R) = \widetilde{\mathcal{C}}.
\]

Now, take an arbitrary $k$-ary operation $f \in \text{Pol}\left(\mathcal{C}^{(m)}\right)$. We examine its induced action $\tilde{f} : A^k \to A$. By Lemma~\ref{lem:action_homomorphism}, for any $g_1, \dots, g_k \in \mathcal{C}^{(m)}$, we have $\tilde{f}(\tilde{g}_1, \dots, \tilde{g}_k) = \widetilde{f(g_1, \dots, g_k)}$. Since $f \in \text{Pol}\left(\mathcal{C}^{(m)}\right)$, the composite operation $f(g_1, \dots, g_k) \in \mathcal{C}^{(m)}$, so its induced action belongs to $\widetilde{\mathcal{C}}^{(m)}$. Thus, $\tilde{f}$ preserves $\widetilde{\mathcal{C}}^{(m)}$.

Therefore, $\tilde{f} \in \text{Pol}(\widetilde{\mathcal{C}}^{(m)}) \subseteq \widetilde{\mathcal{C}}$.

Since $\tilde{f} \in \widetilde{\mathcal{C}}$, there exists some operation $h \in \mathcal{C}$ such that $\tilde{f} = \tilde{h}$. We will explicitly show that $f = h$. 

Since $n \ge |C|$, we can choose a tuple $a = (a_1, \dots, a_n) \in C^n$ that contains every element of $C$. Let $(x_1, \dots, x_k) \in C^k$ be an arbitrary input tuple. For each $i \in \{1, \dots, k\}$, because $x_i \in C$ and $a$ contains all elements of $C$, there is some index $j_i \in \{1, \dots, n\}$ such that $a_{j_i} = x_i$. Because $\mathcal{C}$ is a clone, it contains the $n$-ary projection operations. Let $g_i \in \mathcal{C}^{(n)}$ be the projection operation $\pi_{j_i}^n$. By definition, $g_i(a) = a_{j_i} = x_i$.

Because $\tilde{f} = \tilde{h}$, they must agree when evaluated on the $k$-tuple of operations $(g_1, \dots, g_k)$:
\[
\tilde{f}(g_1, \dots, g_k) = \tilde{h}(g_1, \dots, g_k).
\]
Evaluating both sides as $n$-ary operations on the tuple $a \in C^n$ gives:
\[
\tilde{f}(g_1, \dots, g_k)(a) = \tilde{h}(g_1, \dots, g_k)(a).
\]
By the definition of the induced action, this expands to:
\[
f(g_1(a), \dots, g_k(a)) = h(g_1(a), \dots, g_k(a)).
\]
Substituting $g_i(a) = x_i$, we obtain:
\[
f(x_1, \dots, x_k) = h(x_1, \dots, x_k).
\]
Since $(x_1, \dots, x_k)$ was arbitrary, $f = h$. Because $h \in \mathcal{C}$, this forces $f \in \mathcal{C}$.

Thus, $\text{Pol}\left(\mathcal{C}^{(m)}\right) \subseteq \mathcal{C}$, which implies $\mathcal{C} = \text{Pol}\left(\mathcal{C}^{(m)}\right)$. By Theorem~\ref{thm:finitely_related_nary}, the clone $\mathcal{C}$ is finitely related.
\end{proof}

\begin{corollary}
\label{cor:clone_iso_finitely_related}
Let $\mathcal{C}$ and $\mathcal{D}$ be two isomorphic clones on finite domains $C$ and $D$, respectively. Then $\mathcal{C}$ is finitely related if and only if $\mathcal{D}$ is finitely related.
\end{corollary}

\begin{proof}
Let $\Phi : \mathcal{C} \to \mathcal{D}$ be a clone isomorphism. Formally, $\Phi$ is a bijection that preserves arities, preserves projections:
\[
\Phi(\pi_i^n) = \pi_i^n \quad \text{for all } 1 \le i \le n \text{ and } n \ge 1,
\]
and respects generalized composition, that is, for all $f \in \mathcal{C}^{(k)}$ and $g_1, \dots, g_k \in \mathcal{C}^{(m)}$:
\[
\Phi\big(f(g_1, \dots, g_k)\big) = \Phi(f)\big(\Phi(g_1), \dots, \Phi(g_k)\big).
\]

Choose $n \ge \max(|C|, |D|)$. The isomorphism $\Phi$ restricts to a bijection between the $n$-ary parts $\mathcal{C}^{(n)} \to \mathcal{D}^{(n)}$ and respects the action operations. Consequently, the action algebras $\widetilde{\mathcal{C}} = \mathcal{C} \curvearrowright \mathcal{C}^{(n)}$ and $\widetilde{\mathcal{D}} = \mathcal{D} \curvearrowright \mathcal{D}^{(n)}$ are isomorphic as concrete action clones.

Because $\widetilde{\mathcal{C}} \cong \widetilde{\mathcal{D}}$ as concrete algebras, $\widetilde{\mathcal{C}}$ is finitely related if and only if $\widetilde{\mathcal{D}}$ is finitely related. By Theorem~\ref{thm:action_algebra_finitely_related}:
\begin{align*}
\mathcal{C} \text{ is finitely related} &\iff \widetilde{\mathcal{C}} \text{ is finitely related} \\
&\iff \widetilde{\mathcal{D}} \text{ is finitely related} \\
&\iff \mathcal{D} \text{ is finitely related}.
\end{align*}
Therefore, $\mathcal{C}$ is finitely related if and only if $\mathcal{D}$ is finitely related.
\end{proof}

\subsection{Application to Varieties and Free Algebras}

As an application of Corollary~\ref{cor:clone_iso_finitely_related}, we recover a known finite-relatedness criterion for free algebras (see Davey, Jackson,
Pitkethly, and Szabó \cite{davey2011}). Using natural clone isomorphisms, we show that a finite algebra is finitely related if and only if its corresponding free algebra is finitely related. Before proving this, we must formally define identities, varieties, and term clones.

\begin{definition}[]
\label{def:identities_satisfaction}
Let $t(x_1, \dots, x_m)$ and $s(x_1, \dots, x_m)$ be terms of a given algebraic language. An \emph{identity} is a formal equation denoted by $t \approx s$. An algebra $\mathbf{B}$ is said to \emph{satisfy} this identity if, for every possible assignment of elements $b_1, \dots, b_m \in B$, the corresponding term operations evaluate to the same element in $B$, written $\mathbf{B} \models t \approx s$.
\end{definition}

\begin{definition}[]
\label{def:equational_theory_variety}
The \emph{equational theory} of an algebra $\mathbf{A}$, denoted $\Sigma_{\mathbf{A}}$, is the set of all identities satisfied by $\mathbf{A}$:
\[
\Sigma_{\mathbf{A}} = \{ t \approx s : \mathbf{A} \models t \approx s \}.
\]
The \emph{variety} generated by $\mathbf{A}$, denoted $\mathcal{V}(\mathbf{A})$, is the class of all algebras (of the same signature) that satisfy every identity in $\Sigma_{\mathbf{A}}$:
\[
\mathcal{V}(\mathbf{A}) = \{ \mathbf{B} : \mathbf{B} \models \Sigma_{\mathbf{A}} \}.
\]
\end{definition}

In universal algebra, varieties can be characterized purely through algebraic constructions rather than logical identities, a foundational result proven by Garrett Birkhoff in 1935. 

\begin{theorem}[Birkhoff's HSP Theorem \cite{birkhoff1935, burris1981}]
\label{thm:birkhoff_hsp}
Let $K$ be a class of algebras. We define three closure operators on $K$:
\begin{itemize}
    \item $\mathbf{H}(K)$: The class of all homomorphic images of algebras in $K$.
    \item $\mathbf{S}(K)$: The class of all subalgebras of algebras in $K$.
    \item $\mathbf{P}(K)$: The class of all arbitrary direct products of algebras in $K$.
\end{itemize}
Birkhoff's theorem states that a class of algebras forms a variety if and only if it is closed under homomorphic images, subalgebras, and direct products. Consequently, the variety generated by a single algebra $\mathbf{A}$ can be expressed algebraically as:
\[
\mathcal{V}(\mathbf{A}) = \mathbf{HSP}(\mathbf{A}).
\]
\end{theorem}

\begin{definition}[]
\label{def:free_algebra}
Fix a variety $\mathcal{V} = \mathcal{V}(\mathbf{A})$. The free algebra $\mathbf{F}_n(\mathcal{V})$ in $n$ generators over the variety $\mathcal{V} = \mathcal{V}(\mathbf{A})$ is constructed as the quotient of the set of all $n$-ary terms over the equivalence relation $\sim$, where two terms are equivalent if the algebra $\mathbf{A}$ satisfies the identity:
\[
t \sim s \iff \mathbf{A} \models t \approx s.
\]

The operations on $\mathbf{F}_n(\mathcal{V})$ are induced naturally by term formation. Equivalently, $\mathbf{F}_n(\mathcal{V})$ is uniquely characterized by its universal mapping property: for every algebra $\mathbf{B} \in \mathcal{V}$ and every map $\phi : \{x_1, \dots, x_n\} \to B$, there exists a unique homomorphism $\hat{\phi} : \mathbf{F}_n(\mathcal{V}) \to \mathbf{B}$ that extends $\phi$.
\end{definition}

\begin{definition}[]
\label{def:term_clone}
For any algebra $\mathbf{B}$, the \emph{term clone}, denoted $\text{Clo}(\mathbf{B})$, is the set consisting of all term operations on $\mathbf{B}$.
\end{definition}

\begin{remark}
\label{rem:term_clone_is_clone}
The set $\text{Clo}(\mathbf{B})$ indeed forms a clone on the universe $B$:
\begin{itemize}
    \item It contains all projections $\pi_i^n$, as they correspond to the simple variable terms $x_i$.
    \item It is closed under composition, because replacing variables in a term \linebreak $t(x_1, \dots, x_m)$ with other terms $g_1, \dots, g_m$ yields another valid term
    \[
    t(g_1, \dots, g_m).
    \]
\end{itemize}
\end{remark}

With these structures defined, we can establish a natural isomorphism between the term clone of an algebra and the term clone of its corresponding free algebra.

\begin{proposition}
\label{prop:term_clone_iso}
For $n \ge |A|$, there is a natural clone isomorphism between the term clone of a finite algebra $\mathbf{A}$ and the term clone of its free algebra $\mathbf{F}_n(\mathcal{V}(\mathbf{A}))$. That is:
\[
\text{Clo}(\mathbf{A}) \cong \text{Clo}\big(\mathbf{F}_n(\mathcal{V}(\mathbf{A}))\big).
\]
\end{proposition}

\begin{proof}
We compare the term operations induced by the same term on the two algebras. Define $\Phi : \text{Clo}\big(\mathbf{F}_n(\mathcal{V}(\mathbf{A}))\big) \to \text{Clo}(\mathbf{A})$ by setting $\Phi(t^{\mathbf{F}_n}) = t^{\mathbf{A}}$ for any term $t$.

Since $n \ge |A|$, we can choose $n$ free generators and map them surjectively onto the elements of $A$. By the universal mapping property, this extends to a surjective homomorphism $\mathbf{F}_n(\mathcal{V}(\mathbf{A})) \twoheadrightarrow \mathbf{A}$. Therefore, $\mathbf{A}$ is a homomorphic image of the free algebra, meaning $\mathbf{A} \in \mathbf{H}(\mathbf{F}_n(\mathcal{V}(\mathbf{A})))$. By Birkhoff's HSP Theorem, this implies that the varieties generated by both algebras are identical:
\[
\mathcal{V}(\mathbf{F}_n(\mathcal{V}(\mathbf{A}))) = \mathcal{V}(\mathbf{A}).
\]
Consequently, two terms induce the same operation if and only if the corresponding identity holds in the variety:
\[
t^{\mathbf{F}_n} = s^{\mathbf{F}_n} \iff \mathbf{F}_n(\mathcal{V}(\mathbf{A})) \models t \approx s \iff \mathbf{A} \models t \approx s \iff t^{\mathbf{A}} = s^{\mathbf{A}}.
\]
This logical equivalence simultaneously proves that $\Phi$ is well-defined (the forward direction) and injective (the backward direction).

By definition, every term operation on $\mathbf{A}$ is induced by some term, meaning $\Phi$ is trivially surjective. Finally, because the composition of term operations is again an operation induced by the composition of terms, the mapping $\Phi$ preserves both projections and composition.

Therefore, $\Phi$ is a bijective mapping that preserves the clone structure, making it a clone isomorphism. 
\end{proof}

\begin{theorem}
Let $\mathbf{A}$ be a finite algebra, let $\mathcal{V} = \mathcal{V}(\mathbf{A})$ denote the variety generated by $\mathbf{A}$, and let $\mathbf{F}_n(\mathcal{V})$ be the free algebra of rank $n$ in $\mathcal{V}$. Then, for every $n \ge |A|$,
\[
\mathbf{A} \text{ is finitely related} \iff \mathbf{F}_n(\mathcal{V}) \text{ is finitely related}.
\]
\end{theorem}

\begin{proof}
Recall Corollary~\ref{cor:clone_iso_finitely_related}: if two finite clones are isomorphic, then either both are finitely related or neither is. 

Applying this corollary to the natural clone isomorphism established in Proposition~\ref{prop:term_clone_iso}:
\[
\text{Clo}(\mathbf{A}) \cong \text{Clo}\big(\mathbf{F}_n(\mathcal{V})\big),
\]
we immediately obtain:
\[
\text{Clo}(\mathbf{A}) \text{ is finitely related} \iff \text{Clo}\big(\mathbf{F}_n(\mathcal{V})\big) \text{ is finitely related}.
\]
By definition, an algebra is finitely related precisely when its corresponding term clone is finitely related. Substituting the algebras for their term clones yields the final result:
\[
\mathbf{A} \text{ is finitely related} \iff \mathbf{F}_n(\mathcal{V}) \text{ is finitely related}. \qedhere
\]
\end{proof}

\section*{Acknowledgements}

The author would like to express sincere gratitude to supervisor Prof.~Michael Pinsker for inviting him to participate in research work within the FB1 Algebra Group at TU Wien.

A special thank you goes to scientific advisor Dr.~Žaneta Genčiová for her continuous support, thorough review, and steady guidance. The author also thanks all members of the FB1 Algebra Group for their hospitality during his stay in Vienna.

This summer project has been funded by the European Research Council (Project POCOCOP, ERC Synergy Grant 101071674). Views and opinions expressed are however those of the author only and do not necessarily reflect those of the European Union or the European Research Council Executive Agency. Neither the European Union nor the granting authority can be held responsible for them.


\end{document}